\documentclass[11pt]{article}%
\usepackage{amssymb}
\usepackage{amsmath}
\usepackage{amsfonts}
\usepackage{geometry}
\usepackage{graphicx}%
\providecommand{\U}[1]{\protect\rule{.1in}{.1in}}
\newtheorem{theorem}{Theorem}
\newtheorem{acknowledgement}[theorem]{Acknowledgement}

\newtheorem{definition}[theorem]{Definition}

\newtheorem{lemma}[theorem]{Lemma}

\newtheorem{remark}[theorem]{Remark}

\newenvironment{proof}[1][Proof]{\noindent\textbf{#1.} }{\ \rule{0.5em}{0.5em}}
\begin{document}

\title{Continuity equation in LlogL for the 2D Euler equations under the enstrophy measure}
\author{Giuseppe Da Prato\thanks{Scuola Normale Superiore of Pisa, Italy}, Franco
Flandoli\thanks{Scuola Normale Superiore of Pisa, Italy}, Michael
R\"{o}ckner\thanks{University of Bielefeld, Germany}}
\maketitle

\begin{abstract}
The 2D Euler equations with random initial condition has been investigates by
S. Albeverio and A.-B. Cruzeiro in \cite{AlbCruz} and other authors. Here we
prove existence of solutions for the associated continuity equation in Hilbert
spaces, in a quite general class with LlogL densities with respect to the
enstrophy measure.

\end{abstract}

\section{Introduction}

We consider the 2D Euler equations on the torus $\mathbb{T}^{2}=\mathbb{R}%
^{2}/\mathbb{Z}^{2}$, formulated in terms of the vorticity $\omega$%
\begin{equation}
\partial_{t}\omega+u\cdot\nabla\omega=0 \label{Euler}%
\end{equation}
where $u$ is the velocity, divergence free vector field such that
$\omega=\partial_{2}u_{1}-\partial_{1}u_{2}$. We consider this equation in the
following abstract Wiener space structure. We set $H=L^{2}\left(
\mathbb{T}^{2}\right)  $ with scalar product $\left\langle \cdot
,\cdot\right\rangle _{H}$ and norm $\left\Vert \cdot\right\Vert _{H}$. Given
$\delta>0$, we consider the negative order Sobolev space $B:=H^{-1-\delta
}\left(  \mathbb{T}^{2}\right)  $, its dual $B^{\ast}=H^{1+\delta}\left(
\mathbb{T}^{2}\right)  $, and we write $\left\langle \cdot,\cdot\right\rangle
$ for the dual pairing between elements of $B$ and $B^{\ast}$. More generally,
we shall use the notation $\left\langle \cdot,\cdot\right\rangle $ also for
the dual pairing between elements of $C^{\infty}\left(  \mathbb{T}^{2}\right)
^{\prime}$ and $C^{\infty}\left(  \mathbb{T}^{2}\right)  $; in all cases
$\left\langle \cdot,\cdot\right\rangle $ reduces to $\left\langle \cdot
,\cdot\right\rangle _{H}$ when both elements are in $H$. Let $\mu$ be the so
called "enstrophy measure", the centered Gaussian measure on $B$ (in fact it
is supported on $H^{-1-}\left(  \mathbb{T}^{2}\right)  =\cap_{\delta
>0}H^{-1-\delta}\left(  \mathbb{T}^{2}\right)  $; but not on $H^{-1}\left(
\mathbb{T}^{2}\right)  $) such that
\[
\int_{B}\left\langle \omega,\phi\right\rangle \left\langle \omega
,\psi\right\rangle \mu\left(  d\omega\right)  =\left\langle \phi
,\psi\right\rangle _{H}%
\]
for all $\phi,\psi\in C^{\infty}\left(  \mathbb{T}^{2}\right)  $. Equation
(\ref{Euler}) has been investigated in this framework and it has been proved
that, with a suitable interpretation of the nonlinear term of the equation, it
has a (possibly non unique) solution for $\mu$-almost every initial condition
in $B$. Moreover, on a suitable probability space $\left(  \Xi,\mathcal{F}%
,\mathbb{P}\right)  $, there exists a stationary process with continuous
trajectories in $B$, with marginal law $\mu$ at every time $t$ (in this sense
we could say that $\mu$ is invariant for equation (\ref{Euler}); see also the
infinitesimal invariance \cite{AlbFar}), whose trajectories are solutions of
equation (\ref{Euler}) in that suitable specified sense. These results have
been proved first by Albeverio and Cruzeiro in \cite{AlbCruz} and proved with
a different concept of solution (used below) in \cite{Flandoli}.

We want to study the \textit{continuity equation}, associated to equation
(\ref{Euler}), for a density $\rho_{t}\left(  \omega\right)  $ with respect to
$\mu$. Let us introduce the notation%
\[
b\left(  \omega\right)  =-u\left(  \omega\right)  \cdot\nabla\omega
\]
for the drift in equation (\ref{Euler}), where we stress by writing $u\left(
\omega\right)  $ the fact that $u$ depends on $\omega$. The precise meaning of
$b\left(  \omega\right)  $ is a nontrivial problem discussed below;\ for the
time being, let us take it as an heuristic notation. Let $\mathcal{FC}%
_{b,T}^{1}$ be the set of all functionals $F:\left[  0,T\right]  \times
C^{\infty}\left(  \mathbb{T}^{2}\right)  ^{\prime}\rightarrow\mathbb{R}$ of
the form $F\left(  t,\omega\right)  =\sum_{i=1}^{m}\widetilde{f}_{i}\left(
\left\langle \omega,\phi_{1}\right\rangle ,...,\left\langle \omega,\phi
_{n}\right\rangle \right)  g_{i}\left(  t\right)  $, with $\phi_{1}%
,...,\phi_{n}\in C^{\infty}\left(  \mathbb{T}^{2}\right)  $, $\widetilde{f}%
_{i}\in C_{b}^{1}\left(  \mathbb{R}^{n}\right)  $, $g_{i}\in C^{1}\left(
\left[  0,T\right]  \right)  $ with $g_{i}\left(  T\right)  =0$. The weak form
of the continuity equation is%
\begin{equation}
\int_{0}^{T}\int_{B}\left(  \partial_{t}F\left(  t,\omega\right)
+\left\langle b\left(  \omega\right)  ,DF\left(  t,\omega\right)
\right\rangle \right)  \rho_{t}\left(  \omega\right)  \mu\left(
d\omega\right)  dt=-\int_{B}F\left(  0,\omega\right)  \rho_{0}\left(
\omega\right)  \mu\left(  d\omega\right)  . \label{cont eq weak form}%
\end{equation}
The most critical term, which requires a careful definition, is $\left\langle
b\left(  \omega\right)  ,DF\left(  t,\omega\right)  \right\rangle $. Let us
discuss this issue.

When $F\left(  t,\omega\right)  =\sum_{i=1}^{m}\widetilde{f}_{i}\left(
\left\langle \omega,\phi_{1}\right\rangle ,...,\left\langle \omega,\phi
_{n}\right\rangle \right)  g_{i}\left(  t\right)  $ as above, given any
element $\eta\in C^{\infty}\left(  \mathbb{T}^{2}\right)  ^{\prime}$ the
limit
\[
\lim_{\epsilon\rightarrow0}\epsilon^{-1}\left(  F\left(  t,\omega+\epsilon
\eta\right)  -F\left(  t,\omega\right)  \right)
\]
exists for every $\left(  t,\omega\right)  \in\left[  0,T\right]  \times
C^{\infty}\left(  \mathbb{T}^{2}\right)  ^{\prime}$ and it is equal to
\[
\sum_{i=1}^{m}\sum_{j=1}^{n}\partial_{j}\widetilde{f}_{i}\left(  \left\langle
\omega_{t},\phi_{1}\right\rangle ,...,\left\langle \omega_{t},\phi
_{n}\right\rangle \right)  g_{i}\left(  t\right)  \left\langle \eta,\phi
_{j}\right\rangle .
\]
Assume we have defined $\left\langle b\left(  \omega\right)  ,\phi
\right\rangle $ when $\omega$ is a typical element under $\mu$ and $\phi\in
C^{\infty}\left(  \mathbb{T}^{2}\right)  $. Then we set%
\begin{equation}
\left\langle b\left(  \omega\right)  ,DF\left(  t,\omega\right)  \right\rangle
:=\sum_{i=1}^{m}\sum_{j=1}^{n}\partial_{j}\widetilde{f}_{i}\left(
\left\langle \omega_{t},\phi_{1}\right\rangle ,...,\left\langle \omega
_{t},\phi_{n}\right\rangle \right)  g_{i}\left(  t\right)  \left\langle
b\left(  \omega\right)  ,\phi_{j}\right\rangle . \label{def b}%
\end{equation}
To complete the meaning of $\left\langle b\left(  \omega\right)  ,DF\left(
t,\omega\right)  \right\rangle $ we thus have to give a meaning to
$\left\langle b\left(  \omega\right)  ,\phi\right\rangle $ for every $\phi\in
C^{\infty}\left(  \mathbb{T}^{2}\right)  $. Formally
\[
\left\langle b\left(  \omega\right)  ,\phi\right\rangle =-\left\langle
u\left(  \omega\right)  \cdot\nabla\omega,\phi\right\rangle .
\]
In Theorem \ref{Thm Cauchy} of Section \ref{section def nonlinear term} we
shall define (for each $\phi\in C^{\infty}\left(  \mathbb{T}^{2}\right)  $) a
random variable $\omega\mapsto\left\langle b\left(  \omega\right)
,\phi\right\rangle $ on the space $\left(  B,\mathcal{B},\mu\right)  $
($\mathcal{B}$ being the Borel $\sigma$-field on $B$). With this definition,
identity (\ref{def b}) provides a rigorous definition of the measurable map
$\left(  \omega,t\right)  \mapsto\left\langle b\left(  \omega\right)
,DF\left(  t,\omega\right)  \right\rangle $, with certain integrability
properties in $\omega$ coming from the results of Section
\ref{section def nonlinear term}.

\begin{remark}
To help the intuition, let us heuristically write equation
(\ref{cont eq weak form}) in the form%
\begin{equation}
\partial_{t}\rho_{t}+\operatorname{div}_{\mu}\left(  \rho_{t}b\right)  =0
\label{cont eq}%
\end{equation}
with initial condition $\rho_{0}\left(  \omega\right)  $, where
$\operatorname{div}_{\mu}\left(  v\right)  $, when defined, for a vector field
$v$ on $B$, is (heuristically) defined by the identity%
\begin{equation}
\int_{B}F\left(  \omega\right)  \operatorname{div}_{\mu}\left(  v\left(
\omega\right)  \right)  \mu\left(  d\omega\right)  =-\int_{B}\left\langle
v\left(  \omega\right)  ,DF\left(  \omega\right)  \right\rangle \mu\left(
d\omega\right)  \label{def div mu}%
\end{equation}
for all $F\in\mathcal{FC}_{b}^{1}$, where $\mathcal{FC}_{b}^{1}$ is defined as
$\mathcal{FC}_{b,T}^{1}$ but without the time-dependent components $g_{i}$.
\end{remark}

In \cite{Flandoli} it is proved that the random variable $\omega
\mapsto\left\langle b\left(  \omega\right)  ,\phi\right\rangle $ on $\left(
B,\mathcal{B},\mu\right)  $ has all finite moments;\ here we improve the
result and show that it is exponentially integrable: given $\phi\in C^{\infty
}\left(  \mathbb{T}^{2}\right)  $, it holds
\begin{equation}
\int_{B}e^{\epsilon\left\vert \left\langle b\left(  \omega\right)
,\phi\right\rangle \right\vert }\mu\left(  d\omega\right)  <\infty
\label{exp integr}%
\end{equation}
for some $\epsilon>0$, which depends only on $\left\Vert \phi\right\Vert
_{\infty}$;\ see Theorem \ref{thm exponential est} in Section
\ref{section def nonlinear term} below.

This exponential integrability is a key ingredient to extend, to the 2D Euler
equations, the result of the authors \cite{DaPratoRoeckner} for abstract
equations in Hilbert spaces (in that work the measure $\mu$ is not necessarily
Gaussian, but the nonlinearity is bounded). Indeed, we aim to prove existence
in the class of densities $\rho_{t}\left(  \omega\right)  $ such that
\begin{equation}
\sup_{t\in\left[  0,T\right]  }\int_{B}\rho_{t}\left(  \omega\right)  \log
\rho_{t}\left(  \omega\right)  \mu\left(  d\omega\right)  <\infty.
\label{regularity rho}%
\end{equation}
Since $ab\leq e^{\epsilon a}+\epsilon^{-1}b\left(  \log\epsilon^{-1}%
b-1\right)  $, if $\rho_{t}\left(  \omega\right)  $ satisfies
(\ref{regularity rho}) and property (\ref{exp integr}) is proved, then
\[
\int_{B}\left\langle b\left(  \omega\right)  ,DF\left(  t,\omega\right)
\right\rangle \rho_{t}\left(  \omega\right)  \mu\left(  d\omega\right)
\]
is well defined. With these preliminaries we can give the following definition.

\begin{definition}
Given a measurable function $\rho_{0}:B\rightarrow\lbrack0,\infty)$ such that
$\int_{B}\rho_{0}\left(  \omega\right)  \log\rho_{0}\left(  \omega\right)
\mu\left(  d\omega\right)  <\infty$, we say that a measurable function
$\rho:\left[  0,T\right]  \times B\rightarrow\lbrack0,\infty)$ is a solution
of equation (\ref{cont eq}) of class LlogL if property (\ref{regularity rho})
is satisfied and identity (\ref{cont eq weak form}) holds for every
$F\in\mathcal{FC}_{b,T}^{1}$.
\end{definition}

Our main result, proved in Section \ref{section proof main theorem}, is:

\begin{theorem}
\label{main thm}If
\[
\int_{B}\rho_{0}\left(  \omega\right)  \log\rho_{0}\left(  \omega\right)
\mu\left(  d\omega\right)  <\infty
\]
then there exists a solution of equation (\ref{cont eq}) of class LlogL.
\end{theorem}

\section{Definition and properties of $\left\langle b\left(  \omega\right)
,\phi\right\rangle $ \label{section def nonlinear term}}

We denote by $\left\{  e_{n}\right\}  $ the complete orthonormal system in
$L^{2}\left(  \mathbb{T}^{2};\mathbb{C}\right)  $ given by $e_{n}\left(
x\right)  =e^{2\pi in\cdot x}$, $n\in\mathbb{Z}^{2}$. As already said in the
Introduction, given a distribution $\omega\in C^{\infty}\left(  \mathbb{T}%
^{2}\right)  ^{\prime}$ and a test function $\phi\in C^{\infty}\left(
\mathbb{T}^{2}\right)  $, we denoted by $\left\langle \omega,\phi\right\rangle
$ the duality between $\omega$ and $\phi$ (namely $\omega\left(  \phi\right)
$), and we use the same symbol for the inner product of $L^{2}\left(
\mathbb{T}^{2}\right)  $. We set $\widehat{\omega}\left(  n\right)
=\left\langle \omega,e_{n}\right\rangle $, $n\in\mathbb{Z}^{2}$ and we define,
for each $s\in\mathbb{R}$, the space $H^{s}\left(  \mathbb{T}^{2}\right)  $ as
the space of all distributions $\omega\in C^{\infty}\left(  \mathbb{T}%
^{2}\right)  ^{\prime}$ such that
\[
\left\Vert \omega\right\Vert _{H^{s}}^{2}:=\sum_{n\in\mathbb{Z}^{2}}\left(
1+\left\vert n\right\vert ^{2}\right)  ^{s}\left\vert \widehat{\omega}\left(
n\right)  \right\vert ^{2}<\infty.
\]
We use similar definitions and notations for the space $H^{s}\left(
\mathbb{T}^{2},\mathbb{C}\right)  $ of complex valued functions.

We want to define, for every $\phi\in C^{\infty}\left(  \mathbb{T}^{2}\right)
$, the random variable%
\begin{align*}
\left\langle b\left(  \omega\right)  ,\phi\right\rangle  &  =-\left\langle
u\left(  \omega\right)  \cdot\nabla\omega,\phi\right\rangle =-\int%
_{\mathbb{T}^{2}}u\left(  \omega\right)  \left(  x\right)  \cdot\nabla
\omega\left(  x\right)  \phi\left(  x\right)  dx\\
&  =\int_{\mathbb{T}^{2}}\omega\left(  x\right)  u\left(  \omega\right)
\left(  x\right)  \cdot\nabla\phi\left(  x\right)  dx
\end{align*}
where we have used integration by parts and the condition $\operatorname{div}%
u=0$ (the computation is heuristic, or it holds for smooth periodic functions;
we are still looking for a meaningful definition). Recall that $u$ is
divergence free and associated to $\omega$ by $\omega=\partial_{2}%
u_{1}-\partial_{1}u_{2}$. This relation can be inverted using the so called
Biot-Savart law:%
\[
u\left(  x\right)  =\int_{\mathbb{T}^{2}}K\left(  x-y\right)  \omega\left(
y\right)  dy
\]
where $K\left(  x,y\right)  $ is the Biot-Savart kernel; in full space it is
given by $K\left(  x-y\right)  =\frac{1}{2\pi}\frac{\left(  x-y\right)
^{\perp}}{\left\vert x-y\right\vert ^{2}}$; on the torus its form is less
simple but we still have $K$ smooth for $x\neq y$, $K\left(  y-x\right)
=-K\left(  x-y\right)  $,
\[
\left\vert K\left(  x-y\right)  \right\vert \leq\frac{C}{\left\vert
x-y\right\vert }%
\]
for small values of $\left\vert x-y\right\vert $. See for instance
\cite{Shochet} for details.

The difficulty in the definition of $\left\langle b\left(  \omega\right)
,\phi\right\rangle $ is that $\omega$ is of class $H^{-1-\delta}\left(
\mathbb{T}^{2}\right)  $ and $u$ of class $H^{-\delta}\left(  \mathbb{T}%
^{2}\right)  $, so we need to multiply distributions. The following remark
recalls a trick used in several works on measure-valued solutions of 2D Euler
equations, like \cite{Delort}, \cite{DiPernaMajda}, \cite{Poup},
\cite{Shochet}, \cite{ShochetII}.

\begin{remark}
\label{remark weak vorticity formulation}If $\omega$ is sufficiently smooth
and periodic, using Biot-Savart law we can write%
\[
\left\langle b\left(  \omega\right)  ,\phi\right\rangle =\int_{\mathbb{T}^{2}%
}\int_{\mathbb{T}^{2}}\omega\left(  x\right)  \omega\left(  y\right)  K\left(
x-y\right)  \cdot\nabla\phi\left(  x\right)  dxdy.
\]
Since the double integral, when we rename $x$ by $y$ and $y$ by $x$, is the
same (the renaming doesn't affect the value), and $K\left(  y-x\right)
=-K\left(  x-y\right)  $, we get
\[
\left\langle b\left(  \omega\right)  ,\phi\right\rangle =\int_{\mathbb{T}^{2}%
}\int_{\mathbb{T}^{2}}\omega\left(  x\right)  \omega\left(  y\right)  H_{\phi
}\left(  x,y\right)  dxdy
\]
where%
\[
H_{\phi}\left(  x,y\right)  :=\frac{1}{2}K\left(  x-y\right)  \cdot\left(
\nabla\phi\left(  x\right)  -\nabla\phi\left(  y\right)  \right)  .
\]
The advantage of this symmetrization is that $H_{\phi}$ (opposite to $K\left(
x-y\right)  \cdot\nabla\phi\left(  x\right)  $) is a bounded function. It is
smooth outside the diagonal $x=y$, discontinuous on the diagonal; more
precisely, we can write%
\begin{equation}
H_{\phi}\left(  x,y\right)  =\frac{1}{2\pi}\left\langle D^{2}\phi\left(
x\right)  \frac{x-y}{\left\vert x-y\right\vert },\frac{\left(  x-y\right)
^{\perp}}{\left\vert x-y\right\vert }\right\rangle +R_{\phi}\left(
x,y\right)  \label{decomposition H}%
\end{equation}
where $R_{\phi}\left(  x,y\right)  $ is Lipschitz continuous, with
\[
\left\vert R_{\phi}\left(  x,y\right)  \right\vert \leq C\left\vert
x-y\right\vert .
\]
To summarize, when $\omega$ is sufficiently smooth and periodic, we have%
\[
\left\langle b\left(  \omega\right)  ,\phi\right\rangle =\left\langle
\omega\otimes\omega,H_{\phi}\right\rangle _{L^{2}\left(  \mathbb{T}^{2}%
\times\mathbb{T}^{2}\right)  }%
\]
where $\omega\otimes\omega:\mathbb{T}^{2}\times\mathbb{T}^{2}\rightarrow
\mathbb{R}$ is defined as $\left(  \omega\otimes\omega\right)  \left(
x,y\right)  =\omega\left(  x\right)  \omega\left(  y\right)  $.
\end{remark}

\begin{remark}
The previous expression is meaningful when $\omega$ is a measure, since
$H_{\phi}$ is Borel bounded. When $\omega$ is only a distribution, of class
$H^{-1-\delta}\left(  \mathbb{T}^{2}\right)  $, one can define $\omega
\otimes\omega$ as the unique element of $H^{-2-2\delta}\left(  \mathbb{T}%
^{2}\times\mathbb{T}^{2}\right)  $ such that
\[
\left\langle \omega\otimes\omega,f\right\rangle =\left\langle \omega
,\varphi\right\rangle \left\langle \omega,\psi\right\rangle
\]
for every smooth $f:\mathbb{T}^{2}\times\mathbb{T}^{2}\rightarrow\mathbb{R}$
of the form $f\left(  x,y\right)  =\varphi\left(  x\right)  \psi\left(
y\right)  $, where the dual pairing $\left\langle \omega\otimes\omega
,f\right\rangle $ is on $\mathbb{T}^{2}\times\mathbb{T}^{2}$. But $H_{\phi}$
is not of class $H^{2+2\delta}\left(  \mathbb{T}^{2}\times\mathbb{T}%
^{2}\right)  $, hence there is no simple deterministic meaning for
$\left\langle \omega\otimes\omega,H_{\phi}\right\rangle $ when $\omega\in
H^{-1-\delta}\left(  \mathbb{T}^{2}\right)  $. It is here that probability
will play the essential role.
\end{remark}

In \cite{Flandoli} the following result has been proved. As remarked above,
when $f\in H^{2+2\delta}\left(  \mathbb{T}^{2}\times\mathbb{T}^{2}\right)  $,
$\left\langle \omega\otimes\omega,f\right\rangle $ is well defined for all
$\omega\in H^{-1-\delta}\left(  \mathbb{T}^{2}\right)  $, hence for a.e.
$\omega$ with respect to the Entrophy measure $\mu$.

\begin{lemma}
\label{lemma estimates}Assume $f\in H^{2+\epsilon}\left(  \mathbb{T}^{2}%
\times\mathbb{T}^{2}\right)  $ for some $\epsilon>0$. One has%
\[
\int_{B}\left\vert \left\langle \omega\otimes\omega,f\right\rangle \right\vert
^{p}\mu\left(  d\omega\right)  \leq\frac{\left(  2p\right)  !}{2^{p}%
p!}\left\Vert f\right\Vert _{\infty}^{p}%
\]
for every positive integer $p\geq2$,%
\[
\int_{B}\left\langle \omega\otimes\omega,f\right\rangle \mu\left(
d\omega\right)  =\int_{\mathbb{T}^{2}}f\left(  x,x\right)  dx
\]
and, when $f$ is also symmetric,%
\[
\int_{B}\left\vert \left\langle \omega\otimes\omega,f\right\rangle
-\int_{\mathbb{T}^{2}}f\left(  x,x\right)  dx\right\vert ^{2}\mu\left(
d\omega\right)  =2\int_{\mathbb{T}^{2}}\int_{\mathbb{T}^{2}}f\left(
x,y\right)  ^{2}dxdy.
\]

\end{lemma}

The consequence proved in \cite{Flandoli} is:

\begin{theorem}
\label{Thm Cauchy}Let $\omega:\Xi\rightarrow C^{\infty}\left(  \mathbb{T}%
^{2}\right)  ^{\prime}$ be a white noise and $\phi\in C^{\infty}\left(
\mathbb{T}^{2}\right)  $ be given. Assume that $H_{\phi}^{n}\in H^{2+}\left(
\mathbb{T}^{2}\times\mathbb{T}^{2}\right)  $ are symmetric and approximate
$H_{\phi}$ in the following sense:%
\begin{align*}
\lim_{n\rightarrow\infty}\int_{\mathbb{T}^{2}}\int_{\mathbb{T}^{2}}\left(
H_{\phi}^{n}-H_{\phi}\right)  ^{2}\left(  x,y\right)  dxdy  &  =0\\
\lim_{n\rightarrow\infty}\int_{\mathbb{T}^{2}}H_{\phi}^{n}\left(  x,x\right)
dx  &  =0.
\end{align*}
Then the sequence of r.v.'s $\left\langle \omega\otimes\omega,H_{\phi}%
^{n}\right\rangle $ is a Cauchy sequence in mean square. We denote by
\[
\left\langle b\left(  \omega\right)  ,\phi\right\rangle =\left\langle
\omega\otimes\omega,H_{\phi}\right\rangle
\]
its limit. Moreover, the limit is the same if $H_{\phi}^{n}$ is replaced by
$\widetilde{H}_{\phi}^{n}$ with the same properties and such that
$\lim_{n\rightarrow\infty}\int\int\left(  H_{\phi}^{n}-\widetilde{H}_{\phi
}^{n}\right)  ^{2}\left(  x,y\right)  dxdy=0$.
\end{theorem}

A simple example of functions $H_{\phi}^{n}$ with these properties is given in
\cite{Flandoli}. In addition to these fact, here we prove exponential
integrability, see (\ref{exp integr}).

\begin{theorem}
\label{thm exponential est}Given a bounded measurable $f$ with $\left\Vert
f\right\Vert _{\infty}\leq1$, we have
\[
\int_{B}e^{\epsilon\left\vert \left\langle \omega\otimes\omega,f\right\rangle
\right\vert }\mu\left(  d\omega\right)  <\infty
\]
for all $\epsilon<\frac{1}{2}$.
\end{theorem}

\begin{proof}%
\[
\mathbb{E}\left[  e^{\epsilon\left\vert \left\langle \omega\otimes
\omega,f\right\rangle \right\vert }\right]  =\sum_{p=0}^{\infty}\frac
{\epsilon^{p}\mathbb{E}\left[  \left\vert \left\langle \omega\otimes
\omega,f\right\rangle \right\vert ^{p}\right]  }{p!}\leq\sum_{p=0}^{\infty
}\left(  \frac{\epsilon}{2}\right)  ^{p}\frac{\left(  2p\right)  !}{p!p!}.
\]
This series converges for $\epsilon<\frac{1}{2}$ because (using ratio test)%
\[
\frac{\left(  \frac{\epsilon}{2}\right)  ^{p+1}\frac{\left(  2\left(
p+1\right)  \right)  !}{\left(  p+1\right)  !\left(  p+1\right)  !}}{\left(
\frac{\epsilon}{2}\right)  ^{p}\frac{\left(  2p\right)  !}{p!p!}}%
=\frac{\epsilon}{2}\frac{\left(  2p+2\right)  \left(  2p+1\right)  }{\left(
p+1\right)  \left(  p+1\right)  }\rightarrow2\epsilon\text{.}%
\]

\end{proof}

\section{Proof of Theorem \ref{main thm} \label{section proof main theorem}}

\subsection{Approximate problem}

Recall from the Introduction that $\delta>0$ is fixed and we set $B=$
$H^{-1-\delta}\left(  \mathbb{T}^{2}\right)  $, $H=L^{2}\left(  \mathbb{T}%
^{2}\right)  $; recall also from Section \ref{section def nonlinear term} that
we write $e_{n}\left(  x\right)  =e^{2\pi in\cdot x}$, $x\in\mathbb{T}^{2}$,
$n\in\mathbb{Z}^{2}$, that is a complete orthonormal system in $H^{\mathbb{C}%
}:=L^{2}\left(  \mathbb{T}^{2};\mathbb{C}\right)  $. Given $N\in\mathbb{N}$,
let $H_{N}^{\mathbb{C}}$ be the span of $e_{n}$\ for $\left\vert n\right\vert
_{\infty}\leq N$, $\left\vert n\right\vert _{\infty}:=\max\left(  \left\vert
n_{1}\right\vert ,\left\vert n_{2}\right\vert \right)  $ for $n=\left(
n_{1},n_{2}\right)  $; it is a subspace of $H^{\mathbb{C}}$. Let $H_{N}$ be
the subspace of $H_{N}^{\mathbb{C}}$ made of real-valued elements; it is a
subspace of $H$ and is characterized by the following property: $\omega
=\sum_{\left\vert n\right\vert _{\infty}\leq N}\omega_{n}e_{n}$ is in $H_{N}$
if and only if $\overline{\omega_{n}}=\omega_{-n}$, for all $n$ such that
$\left\vert n\right\vert _{\infty}\leq N$.

Let $\pi_{N}$ be the orthogonal projection of $H$ onto $H_{N}$. It is given by
$\pi_{N}\omega=\sum_{\left\vert n\right\vert _{\infty}\leq N}\left\langle
\omega,e_{n}\right\rangle _{H}e_{n}$, for all $\omega\in H$. We extend
$\pi_{N}$ to an operator on $B$ by setting%
\begin{align*}
\pi_{N} &  :B\rightarrow H_{N}\\
\pi_{N}\omega &  =\sum_{\left\vert n\right\vert _{\infty}\leq N}\left\langle
\omega,e_{n}\right\rangle e_{n}%
\end{align*}
where now $\left\langle \omega,e_{n}\right\rangle $ is the dual pairing. We
may introduce the Dirichlet kernel%
\begin{equation}
\theta_{N}\left(  x_{1},x_{2}\right)  =\sum_{n_{1}=-N}^{N}\sum_{n_{2}=-N}%
^{N}e^{2\pi i\left(  n_{1}x_{1}+n_{2}x_{2}\right)  }=\sum_{\left\vert
n\right\vert _{\infty}\leq N}e^{2\pi in\cdot x}\label{Dirichlet kernel}%
\end{equation}
for $x=\left(  x_{1},x_{2}\right)  \in\mathbb{T}^{2}$, and check that%
\[
\pi_{N}\omega=\theta_{N}\ast\omega.
\]

We define the operator
\[
b_{N}:B\rightarrow H_{N}%
\]
as%
\[
b_{N}\left(  \omega\right)  =-\pi_{N}\left(  u\left(  \pi_{N}\omega\right)
\cdot\nabla\pi_{N}\omega\right)  ,\qquad\omega\in B
\]
where $u\left(  \pi_{N}\omega\right)  $ denotes the result of Biot-Savart law
applied to $\pi_{N}\omega$,
\[
u\left(  \pi_{N}\omega\right)  \left(  x\right)  :=\int_{\mathbb{T}^{2}%
}K\left(  x-y\right)  \left(  \pi_{N}\omega\right)  \left(  y\right)  dy.
\]
The operator $b_{N}$ has the following properties. We denote by
$\operatorname{div}b_{N}\left(  \omega\right)  $ the function%
\[
\operatorname{div}b_{N}\left(  \omega\right)  =\sum_{\left\vert n\right\vert
_{\infty}\leq N}\partial_{n}\left\langle b_{N}\left(  \omega\right)
,e_{n}\right\rangle _{H}%
\]
where, when defined, $\partial_{n}F\left(  \omega\right)  =\lim_{\epsilon
\rightarrow0}\epsilon^{-1}\left(  F\left(  \omega+\epsilon e_{n}\right)
-F\left(  \omega\right)  \right)  $, for a function $F$ defined on $B$. We say
that $\operatorname{div}b_{N}\left(  \omega\right)  $ exists if $\partial
_{n}\left\langle b_{N}\left(  \omega\right)  ,e_{n}\right\rangle _{H}$ exists
for all $\left\vert n\right\vert _{\infty}\leq N$. Moreover, we set%
\[
\operatorname{div}_{\mu}b_{N}\left(  \omega\right)  :=\operatorname{div}%
b_{N}\left(  \omega\right)  -\left\langle \omega,b_{N}\left(  \omega\right)
\right\rangle
\]
where $\left\langle \omega,b_{N}\left(  \omega\right)  \right\rangle $ is the
dual pairing. It is easy to check that this definition is coherent with the
general one (\ref{def div mu}) given in the Introduction.

\begin{lemma}
\label{lemma div b}The divergence $\operatorname{div}b_{N}\left(
\omega\right)  $ exists for all $\omega\in B$ and
\begin{align*}
\operatorname{div}b_{N}\left(  \omega\right)   &  =0\\
\left\langle \omega,b_{N}\left(  \omega\right)  \right\rangle  &  =0
\end{align*}
and thus%
\[
\operatorname{div}_{\mu}b_{N}\left(  \omega\right)  =0.
\]

\end{lemma}

\begin{proof}
\textbf{Step 1}: A basic identity is%
\[
\left\langle \omega,b_{N}\left(  \omega\right)  \right\rangle =0
\]
for all $\omega\in B$, where as usual $\left\langle .,.\right\rangle $ denotes
dual pairing. This identity holds because%
\[
\left\langle \omega,\pi_{N}\left(  u\left(  \pi_{N}\omega\right)  \cdot
\nabla\pi_{N}\omega\right)  \right\rangle =\left\langle \pi_{N}\omega,u\left(
\pi_{N}\omega\right)  \cdot\nabla\pi_{N}\omega\right\rangle _{H}=0
\]
where the first equality can be checked by writing $\omega=\sum\left\langle
\omega,e_{n}\right\rangle e_{n}$ (the series converges in $B$), and the second
equality is true because
\[
\left\langle v\cdot\nabla f,f\right\rangle =\frac{1}{2}\int_{\mathbb{T}^{2}%
}v\left(  x\right)  \cdot\nabla f^{2}\left(  x\right)  dx=-\frac{1}{2}%
\int_{\mathbb{T}^{2}}\operatorname{div}v\left(  x\right)  f^{2}\left(
x\right)  dx=0
\]
for all sufficiently smooth divergence free vector field $v$ (we take
$v=u\left(  \pi_{N}\omega\right)  $ that is a smooth divergence free vector
field) and all sufficiently smooth functions $f$ (we take $f=\pi_{N}\omega$).

\textbf{Step 2}: Recall that $u\left(  e_{n}\right)  \left(  x\right)  $ is
periodic, divergence free, and such that $\nabla^{\perp}\cdot u\left(
e_{n}\right)  =e_{n}$ (it is also gven by the Biot-Savart law $u\left(
e_{n}\right)  \left(  x\right)  :=\int_{\mathbb{T}^{2}}K\left(  x-y\right)
e_{n}\left(  y\right)  dy$). Then we have%
\[
u\left(  e_{n}\right)  \left(  x\right)  \cdot\nabla e_{n}\left(  x\right)  =0
\]
for every $n\in\mathbb{Z}^{2}$. Indeed,
\[
u\left(  e_{n}\right)  \left(  x\right)  \cdot\nabla e_{n}\left(  x\right)
=2\pi i\left(  u\left(  e_{n}\right)  \left(  x\right)  \cdot n\right)
e_{n}\left(  x\right)
\]
and this is zero because $u\left(  e_{n}\right)  \left(  x\right)  \cdot n=0$.
To prove the latter property, it is necessary to understand the shape of
$u\left(  e_{n}\right)  \left(  x\right)  $. Let us prove that
\[
u\left(  e_{n}\right)  \left(  x\right)  =\frac{n^{\perp}}{\left\vert
n\right\vert ^{2}}e_{n}\left(  x\right)
\]
(which implies $u\left(  e_{n}\right)  \left(  x\right)  \cdot n=0$ because
$n^{\perp}\cdot n=0$). The function $u\left(  e_{n}\right)  $ is uniquely
defined by the conditions to be periodic, divengence free and $\nabla^{\perp
}\cdot u\left(  e_{n}\right)  =e_{n}$, so we have to check these conditions
for the function $\frac{n^{\perp}}{\left\vert n\right\vert ^{2}}e_{n}\left(
x\right)  $. This is clearly periodic; it is divengence free because
$\operatorname{div}u\left(  e_{n}\right)  \left(  x\right)  =\frac{n^{\perp}%
}{\left\vert n\right\vert ^{2}}e_{n}\left(  x\right)  \cdot n=0$; and finally
$\nabla^{\perp}\cdot\frac{n^{\perp}}{\left\vert n\right\vert ^{2}}e_{n}\left(
x\right)  =\frac{n^{\perp}}{\left\vert n\right\vert ^{2}}e_{n}\left(
x\right)  \cdot n^{\perp}=e_{n}\left(  x\right)  $.

\textbf{Step 3}: Finally we can prove that $\operatorname{div}b_{N}\left(
\omega\right)  =0$. It is
\[
\operatorname{div}b_{N}\left(  \omega\right)  =-\sum_{\left\vert n\right\vert
\leq N}\partial_{n}\left\langle \pi_{N}\left(  u\left(  \pi_{N}\omega\right)
\cdot\nabla\pi_{N}\omega\right)  ,e_{n}\right\rangle _{H}.
\]
We have
\begin{align*}
&  \partial_{n}\left\langle \pi_{N}\left(  u\left(  \pi_{N}\omega\right)
\cdot\nabla\pi_{N}\omega\right)  ,e_{n}\right\rangle _{H}\\
&  =\partial_{n}\left\langle u\left(  \pi_{N}\omega\right)  \cdot\nabla\pi
_{N}\omega,e_{n}\right\rangle _{H}\\
&  =-\partial_{n}\left\langle \pi_{N}\omega,u\left(  \pi_{N}\omega\right)
\cdot\nabla e_{n}\right\rangle _{H}%
\end{align*}
(we have used integration by parts and $\operatorname{div}u\left(  \pi
_{N}\omega\right)  =0$ in the last identity)%
\begin{align*}
& =-\left\langle \partial_{n}\left(  \pi_{N}\omega\right)  ,u\left(  \pi
_{N}\omega\right)  \cdot\nabla e_{n}\right\rangle _{H}-\left\langle \pi
_{N}\omega,\partial_{n}\left(  u\left(  \pi_{N}\omega\right)  \cdot\nabla
e_{n}\right)  \right\rangle _{H}\\
& =-\left\langle e_{n},u\left(  \pi_{N}\omega\right)  \cdot\nabla
e_{n}\right\rangle _{H}-\left\langle \pi_{N}\omega,u\left(  e_{n}\right)
\cdot\nabla e_{n}\right\rangle _{H}%
\end{align*}
because%
\[
\partial_{n}\left(  \pi_{N}\omega\right)  =\partial_{n}\left(  \sum
_{\left\vert n^{\prime}\right\vert \leq N}\left\langle \omega,e_{n^{\prime}%
}\right\rangle e_{n^{\prime}}\right)  =\sum_{\left\vert n^{\prime}\right\vert
\leq N}\partial_{n}\left(  \left\langle \omega,e_{n^{\prime}}\right\rangle
\right)  e_{n^{\prime}}=\sum_{\left\vert n^{\prime}\right\vert \leq N}%
\delta_{nn^{\prime}}e_{n^{\prime}}%
\]%
\[
\partial_{n}\left(  u\left(  \pi_{N}\omega\right)  \cdot\nabla e_{n}\right)
=\partial_{n}\left(  \sum_{\left\vert n^{\prime\prime}\right\vert \leq
N}\left\langle \omega,e_{n^{\prime\prime}}\right\rangle u\left(
e_{n^{\prime\prime}}\right)  \cdot\nabla e_{n}\right)  =\sum_{\left\vert
n^{\prime\prime}\right\vert \leq N}\delta_{nn^{\prime\prime}}u\left(
e_{n^{\prime\prime}}\right)  \cdot\nabla e_{n}.
\]
The first term, $\left\langle e_{n},u\left(  \pi_{N}\omega\right)  \cdot\nabla
e_{n}\right\rangle _{H}$, is zero by the same general rule recalled in Step 1.
The second term is zero by Step 2. Therefore $\operatorname{div}b_{N}\left(
\omega\right)  =0$.
\end{proof}

\vspace{0.5cm}

Consider the finite dimensional ordinary differential equation in the space
$H_{N}$ defined as%
\begin{equation}
\frac{d\omega_{t}^{N}}{dt}=b_{N}\left(  \omega_{t}^{N}\right)  ,\qquad
\omega_{0}^{N}\in H_{N}. \label{approx probl}%
\end{equation}
The function $b_{N}$, in $H_{N}$, is differentable, bounded with bounded
derivative on bounded sets. Hence, for every $\omega_{0}^{N}\in H_{N}$, there
is a unique local solution $\omega_{t}^{N,\omega_{0}^{N}}$ of equation
(\ref{approx probl}) and the flow map $\omega_{0}^{N}\mapsto\omega
_{t}^{N,\omega_{0}^{N}}$, where defined, is continuously differentiable,
invertible with continuously differentiable inverse. The solution is global
because of the energy estimate%
\[
\frac{d\left\Vert \omega_{t}^{N}\right\Vert _{H}^{2}}{dt}=2\left\langle
b_{N}\left(  \omega_{t}^{N}\right)  ,\omega_{t}^{N}\right\rangle _{H}=0
\]
which implies $\sup_{t\in\left[  0,\tau\right]  }\left\Vert \omega_{t}%
^{N}\right\Vert _{H}^{2}\leq\left\Vert \omega_{0}^{N}\right\Vert _{H}^{2}$ on
any interval $\left[  0,\tau\right]  $ of local existence; the property
$\left\langle b_{N}\left(  \omega_{t}^{N}\right)  ,\omega_{t}^{N}\right\rangle
_{H}=0$ holds by Lemma \ref{lemma div b}. We denote by $\Phi_{t}^{N}%
:H_{N}\rightarrow H_{N}$ the global flow defined as $\Phi_{t}^{N}\left(
\omega_{0}^{N}\right)  =\omega_{t}^{N,\omega_{0}^{N}}$.

Denote by $\mu^{N}\left(  d\omega\right)  $ the image measure, on $H_{N}$, of
$\mu\left(  d\omega\right)  $ under the projection $\pi_{N}$. This measure is
invariant under the flow $\Phi_{t}^{N}$, because $\operatorname{div}_{\mu
}b_{N}\left(  \omega\right)  =0$: for every smooth $F:H_{N}\rightarrow
\lbrack0,\infty)$, bounded with bounded derivatives,
\begin{align*}
\int_{H_{N}}\left\langle b_{N}\left(  \omega\right)  ,DF\left(  \omega\right)
\right\rangle _{H_{N}}\mu^{N}\left(  d\omega\right)   &  =\int_{B}\left\langle
b_{N}\left(  \omega\right)  ,DF\left(  \pi_{N}\omega\right)  \right\rangle
_{H}\mu\left(  d\omega\right) \\
&  =-\int_{B}F\left(  \pi_{N}\omega\right)  \operatorname{div}_{\mu}%
b_{N}\left(  \omega\right)  \mu\left(  d\omega\right)  =0.
\end{align*}

\subsection{Continuity equation for the approximate problem}

Given a measurable function $\rho_{0}^{N}:H_{N}\rightarrow\lbrack0,\infty)$,
with $\int_{B}\rho_{0}^{N}\left(  \pi_{N}\omega\right)  \mu\left(
d\omega\right)  <\infty$, consider the measure $\rho_{0}^{N}\left(  \pi
_{N}\omega\right)  \mu^{N}\left(  d\omega\right)  $ and its push forward under
the flow map $\Phi_{t}^{N}$; denote it by $\nu_{t}^{N}$. By definition, for
bounded measurable $F:H_{N}\rightarrow\lbrack0,\infty)$,
\[
\int_{H_{N}}F\left(  \omega\right)  \nu_{t}^{N}\left(  d\omega\right)
=\int_{H_{N}}F\left(  \Phi_{t}^{N}\left(  \omega\right)  \right)  \rho_{0}%
^{N}\left(  \omega\right)  \mu^{N}\left(  d\omega\right)  .
\]
From the invariance of $\mu^{N}$ under the flow $\Phi_{t}^{N}$, we have%
\[
\int_{H_{N}}F\left(  \omega\right)  \nu_{t}^{N}\left(  d\omega\right)
=\int_{H_{N}}F\left(  \omega\right)  \rho_{0}^{N}\left(  \left(  \Phi_{t}%
^{N}\right)  ^{-1}\left(  \omega\right)  \right)  \mu^{N}\left(
d\omega\right)
\]
hence%
\[
\nu_{t}^{N}\left(  d\omega\right)  =\rho_{t}^{N}\left(  \pi_{N}\omega\right)
\mu^{N}\left(  d\omega\right)
\]
where
\begin{equation}
\rho_{t}^{N}\left(  \omega\right)  =\rho_{0}^{N}\left(  \left(  \Phi_{t}%
^{N}\right)  ^{-1}\left(  \omega\right)  \right)  ,\qquad\omega\in H_{N}.
\label{formula for density}%
\end{equation}
We have partially proved the following statement.

\begin{lemma}
\label{lemma rho n}Consider equation (\ref{approx probl}) in $H_{N}$, with the
associated flow $\Phi_{t}^{N}$. Given at time zero a measure of the form
$\rho_{0}^{N}\left(  \pi_{N}\omega\right)  \mu^{N}\left(  d\omega\right)  $
with $\int_{B}\rho_{0}^{N}\left(  \pi_{N}\omega\right)  \mu\left(
d\omega\right)  <\infty$, its push forward at time $t$, under the flow map
$\Phi_{t}^{N}$, is a measure of the form $\rho_{t}^{N}\left(  \pi_{N}%
\omega\right)  \mu^{N}\left(  d\omega\right)  $, with $\int_{B}\rho_{t}%
^{N}\left(  \pi_{N}\omega\right)  \mu\left(  d\omega\right)  <\infty$. If in
addition $\int_{B}\rho_{0}^{N}\left(  \pi_{N}\omega\right)  \log\rho_{0}%
^{N}\left(  \pi_{N}\omega\right)  \mu\left(  d\omega\right)  <\infty$, the
same is true at time $t$ and%
\begin{equation}
\int_{B}\rho_{t}^{N}\left(  \pi_{N}\omega\right)  \log\rho_{t}^{N}\left(
\pi_{N}\omega\right)  \mu\left(  d\omega\right)  =\int_{B}\rho_{0}^{N}\left(
\pi_{N}\omega\right)  \log\rho_{0}^{N}\left(  \pi_{N}\omega\right)  \mu\left(
d\omega\right)  . \label{LlogL for approx}%
\end{equation}
If in addition $\rho_{0}^{N}$ is bounded, then $\rho_{t}^{N}\leq\left\Vert
\rho_{0}^{N}\right\Vert _{\infty}$. Finally. $\rho_{t}^{N}$ satisfies the
continuity equation%
\begin{equation}
\int_{0}^{T}\int_{B}\left(  \partial_{t}F\left(  t,\omega\right)
+\left\langle DF\left(  t,\omega\right)  ,b_{N}\left(  \omega\right)
\right\rangle _{H}\right)  \rho_{t}^{N}\left(  \pi_{N}\omega\right)
\mu\left(  d\omega\right)  dt=-\int_{B}F\left(  0,\omega\right)  \rho_{0}%
^{N}\left(  \pi_{N}\omega\right)  \mu\left(  d\omega\right)
\label{cont weak form approx}%
\end{equation}
for all $F\in\mathcal{FC}_{b,T}^{1}$ of the form $F\left(  t,\omega\right)
=\sum_{i=1}^{m}\widetilde{f}_{i}\left(  \left\langle \omega,e_{n}\right\rangle
,\left\vert n\right\vert _{\infty}\leq N\right)  g_{i}\left(  t\right)  $.
\end{lemma}

\begin{proof}
The integrability of $\rho_{t}^{N}$ comes from the invariance of $\mu^{N}$
under $\Phi_{t}^{N}$, as well as the LlogL property; let us check this latter
one. Using (\ref{formula for density}) we have%
\begin{align*}
\int_{B}\rho_{t}^{N}\left(  \pi_{N}\omega\right)  \log\rho_{t}^{N}\left(
\pi_{N}\omega\right)  \mu\left(  d\omega\right)   &  =\int_{H_{N}}\rho_{t}%
^{N}\left(  \omega\right)  \log\rho_{t}^{N}\left(  \omega\right)  \mu
^{N}\left(  d\omega\right) \\
&  =\int_{H_{N}}\rho_{0}^{N}\left(  \left(  \Phi_{t}^{N}\right)  ^{-1}\left(
\omega\right)  \right)  \log\rho_{0}^{N}\left(  \left(  \Phi_{t}^{N}\right)
^{-1}\left(  \omega\right)  \right)  \mu^{N}\left(  d\omega\right) \\
&  =\int_{H_{N}}\rho_{0}^{N}\left(  \omega\right)  \log\rho_{0}^{N}\left(
\omega\right)  \mu^{N}\left(  d\omega\right) \\
&  =\int_{B}\rho_{0}^{N}\left(  \pi_{N}\omega\right)  \log\rho_{0}^{N}\left(
\pi_{N}\omega\right)  \mu\left(  d\omega\right)  .
\end{align*}

When $\rho_{0}^{N}$ is bounded, we have
\[
\rho_{t}^{N}\left(  \omega\right)  =\rho_{0}^{N}\left(  \left(  \Phi_{t}%
^{N}\right)  ^{-1}\left(  \omega\right)  \right)  \leq\left\Vert \rho_{0}%
^{N}\right\Vert _{\infty}.
\]

Finally, from the chain rule applied to $F\left(  t,\Phi_{t}^{N}\left(
\omega\right)  \right)  $, $\omega\in H_{N}$, we get the weak form of the
continuity equation.
\end{proof}

\begin{remark}
We may construct $\rho_{t}^{N}$ and prove (\ref{LlogL for approx}) also by the
following procedure, closer to \cite{DaPratoRoeckner}. We study the transport
equation in $H_{N}$
\[
\partial_{t}\rho_{t}^{N}+\left\langle b_{N},D\rho_{t}^{N}\right\rangle _{H}=0
\]
with initial condition $\rho_{0}^{N}$, which has the solution
(\ref{formula for density}) by the method of characteristics. Its weak form
reduces to (\ref{cont weak form approx}) because (for $F$ like those of the
Lemma)%
\begin{align*}
&  \int_{H_{N}}F\left(  t,\omega\right)  \left\langle b_{N}\left(
\omega\right)  ,D\rho_{t}^{N}\left(  \omega\right)  \right\rangle _{H}\mu
^{N}\left(  d\omega\right) \\
&  =\int_{B}F\left(  t,\omega\right)  \left\langle b_{N}\left(  \omega\right)
,D\rho_{t}^{N}\left(  \pi_{N}\omega\right)  \right\rangle _{H}\mu\left(
d\omega\right) \\
&  =-\int_{B}\left\langle DF\left(  t,\omega\right)  ,b_{N}\left(
\omega\right)  \right\rangle _{H}\rho_{t}^{N}\left(  \pi_{N}\omega\right)
\mu\left(  d\omega\right)
\end{align*}
where we have used the property $\operatorname{div}_{\mu}b_{N}\left(
\omega\right)  =0$. Finally, to prove (\ref{LlogL for approx}) as in
\cite{DaPratoRoeckner}, we compute
\begin{align*}
&  \frac{d}{dt}\int_{H_{N}}\rho_{t}^{N}\left(  \log\rho_{t}^{N}-1\right)
d\mu^{N}\\
&  =\int_{H_{N}}\log\rho_{t}^{N}\partial_{t}\rho_{t}^{N}d\mu^{N}=-\int_{H_{N}%
}\log\rho_{t}^{N}\left\langle b_{N},D\rho_{t}^{N}\right\rangle d\mu^{N}\\
&  =-\int_{H_{N}}\left\langle b_{N},D\left[  \rho_{t}^{N}\left(  \log\rho
_{t}^{N}-1\right)  \right]  \right\rangle d\mu^{N}\\
&  =\int_{H_{N}}\left[  \rho_{t}^{N}\left(  \log\rho_{t}^{N}-1\right)
\right]  \operatorname{div}_{\mu}b_{N}d\mu^{N}=0.
\end{align*}

\end{remark}

\subsection{Construction of a solution to the limit problem}

\subsubsection{First case:\ bounded continuous $\rho_{0}$}

Consider first the case when $\rho_{0}$ is a bounded continuous function on
$B$. Define the sequence of equibounded functions $\rho_{0}^{N}$ on $H_{N}$ by
setting $\rho_{0}^{N}\left(  \pi_{N}\omega\right)  =\rho_{0}\left(  \pi
_{N}\omega\right)  $. For each one of them, consider the associated function
$\rho_{t}^{N}\left(  \pi_{N}\omega\right)  $ given by Lemma \ref{lemma rho n}.
There is a subsequence, still denoted for simplicity by $\rho_{t}^{N}\left(
\pi_{N}\omega\right)  $ which converges to some function $\rho_{t}$ weak* in
$L^{\infty}\left(  \left[  0,T\right]  \times B\right)  $; entropy is weakly
lower semicontinuous in $L^{1}\left(  B,\mu\right)  $, hence
\[
\int_{B}\rho_{t}\left(  \omega\right)  \log\rho_{t}\left(  \omega\right)
\mu\left(  d\omega\right)  \leq\underset{N\rightarrow\infty}{\lim\inf}\int%
_{B}\rho_{t}^{N}\left(  \pi_{N}\omega\right)  \log\rho_{t}^{N}\left(  \pi
_{N}\omega\right)  \mu\left(  d\omega\right)  .
\]
By (\ref{LlogL for approx}) we deduce%
\[
\int_{B}\rho_{t}\left(  \omega\right)  \log\rho_{t}\left(  \omega\right)
\mu\left(  d\omega\right)  \leq\underset{N\rightarrow\infty}{\lim\inf}\int%
_{B}\rho_{0}^{N}\left(  \pi_{N}\omega\right)  \log\rho_{0}^{N}\left(  \pi
_{N}\omega\right)  \mu\left(  d\omega\right)  .
\]
But, by the definition above of $\rho_{0}^{N}$,
\[
\int_{B}\rho_{0}^{N}\left(  \pi_{N}\omega\right)  \log\rho_{0}^{N}\left(
\pi_{N}\omega\right)  \mu\left(  d\omega\right)  =\int_{B}\rho_{0}\left(
\pi_{N}\omega\right)  \log\rho_{0}\left(  \pi_{N}\omega\right)  \mu\left(
d\omega\right)  .
\]
Using Lebesgue dominated convergence theorem, this finally implies, by
continuity of $\rho_{0}$ and of the function $x\log x$, and by boundedness of
$\rho_{0}$,
\[
\int_{B}\rho_{t}\left(  \omega\right)  \log\rho_{t}\left(  \omega\right)
\mu\left(  d\omega\right)  \leq\int_{B}\rho_{0}\left(  \omega\right)  \log
\rho_{0}\left(  \omega\right)  \mu\left(  d\omega\right)  .
\]

Finally we have to prove that $\rho_{t}$ satisfies the weak formulation. We
have to pass to the limit in (\ref{cont weak form approx}). The only problem
is the term
\[
\int_{0}^{T}\int_{B}\left\langle b_{N}\left(  \omega\right)  ,DF\left(
t,\omega\right)  \right\rangle _{H}\rho_{t}^{N}\left(  \pi_{N}\omega\right)
\mu\left(  d\omega\right)  dt.
\]
We add and subtract the term%
\[
\int_{0}^{T}\int_{B}\left\langle b\left(  \omega\right)  ,DF\left(
t,\omega\right)  \right\rangle \rho_{t}^{N}\left(  \pi_{N}\omega\right)
\mu\left(  d\omega\right)  dt
\]
and use integrability of $\left\langle b\left(  \omega\right)  ,D_{H}F\left(
t,\omega\right)  \right\rangle $ and weak* convergence of $\rho_{t}^{N}\left(
\pi_{N}\omega\right)  $ to $\rho_{t}\left(  \omega\right)  $ to pass to the
limit in one addend. It remains to prove that
\[
\lim_{N\rightarrow\infty}\int_{0}^{T}\int_{B}\left(  \left\langle b_{N}\left(
\omega\right)  ,DF\left(  t,\omega\right)  \right\rangle _{H}-\left\langle
b\left(  \omega\right)  ,DF\left(  t,\omega\right)  \right\rangle \right)
\rho_{t}^{N}\left(  \pi_{N}\omega\right)  \mu\left(  d\omega\right)  dt=0.
\]
Keeping in mind again the weak* convergence of $\rho_{t}^{N}\left(  \pi
_{N}\omega\right)  $, it is sufficient to prove that $\int_{B}\left\langle
b_{N}\left(  \omega\right)  ,D_{H}F\left(  t,\omega\right)  \right\rangle
_{H}$ converges strongly to $\left\langle b\left(  \omega\right)
,D_{H}F\left(  t,\omega\right)  \right\rangle $ in $L^{1}\left(
0,T;L^{1}\left(  B,\mu\right)  \right)  $. Due to the form of $F$, it is
sufficient to prove the following claim:\ given $\phi\in C^{\infty}\left(
\mathbb{T}^{2}\right)  $,
\[
\lim_{N\rightarrow\infty}\int_{B}\left\vert \left\langle b_{N}\left(
\omega\right)  ,\phi\right\rangle _{H}-\left\langle b\left(  \omega\right)
,\phi\right\rangle \right\vert \mu\left(  d\omega\right)  =0.
\]
The remainder of this subsection is devoted to the proof of this claim.

It is not restrictive to assume that $\phi\in H_{N_{0}}$ for some $N_{0}$.
Hence, for $N$ large enough so that $\pi_{N}\phi=\phi$,
\begin{align*}
\left\langle b_{N}\left(  \omega\right)  ,\phi\right\rangle _{H} &
=-\left\langle \pi_{N}\left(  u\left(  \pi_{N}\omega\right)  \cdot\nabla
\pi_{N}\omega\right)  ,\phi\right\rangle _{H}\\
&  =-\left\langle u\left(  \pi_{N}\omega\right)  \cdot\nabla\pi_{N}\omega
,\phi\right\rangle _{H}\\
&  =\left\langle \pi_{N}\omega,u\left(  \pi_{N}\omega\right)  \cdot\nabla
\phi\right\rangle _{H}\\
&  =\left\langle \left(  \pi_{N}\omega\right)  \otimes\left(  \pi_{N}%
\omega\right)  ,H_{\phi}\right\rangle
\end{align*}
where the last identity is proved as in Remark
\ref{remark weak vorticity formulation}. We have%
\[
\left\langle \left(  \pi_{N}\omega\right)  \otimes\left(  \pi_{N}%
\omega\right)  ,H_{\phi}\right\rangle =\left\langle \omega\otimes
\omega,\left(  H_{\phi}\right)  _{N}\right\rangle
\]
where%
\[
\left(  H_{\phi}\right)  _{N}\left(  x,y\right)  =\sum_{\left\vert
n\right\vert _{\infty}\leq N}\sum_{\left\vert n^{\prime}\right\vert _{\infty
}\leq N}e_{n}\left(  x\right)  e_{n^{\prime}}\left(  y\right)  \int%
_{\mathbb{T}^{2}}\int_{\mathbb{T}^{2}}e_{n^{\prime}}\left(  y^{\prime}\right)
e_{n}\left(  x^{\prime}\right)  H_{\phi}\left(  x^{\prime},y^{\prime}\right)
dx^{\prime}dy^{\prime}.
\]
Therefore, our aim is to prove that, given $\phi\in C^{\infty}\left(
\mathbb{T}^{2}\right)  $,
\[
\lim_{N\rightarrow\infty}\int_{B}\left\vert \left\langle \omega\otimes
\omega,\left(  H_{\phi}\right)  _{N}-H_{\phi}\right\rangle \right\vert
\mu\left(  d\omega\right)  =0.
\]
Thanks to Lemma \ref{lemma estimates} and Theorem \ref{Thm Cauchy}, with a
simple argument on Cauchy sequences one can see that it is sufficient to prove
that $\left(  H_{\phi}\right)  _{N}\rightarrow H_{\phi}$ in $L^{2}\left(
\mathbb{T}^{2}\times\mathbb{T}^{2}\right)  $ and
\begin{equation}
\int_{\mathbb{T}^{2}}\left(  H_{\phi}\right)  _{N}\left(  x,x\right)
dx\rightarrow0.\label{last limit property}%
\end{equation}
From the theory of Fourier series, $\left(  H_{\phi}\right)  _{N}\rightarrow
H_{\phi}$ in $L^{2}\left(  \mathbb{T}^{2}\times\mathbb{T}^{2}\right)  $. The
limit property (\ref{last limit property}) requires more work. The result is
included in the next lemma, which completes the proof that $\rho_{t}$ is a
weak solution, in the case when $\rho_{0}$ is bounded.

\begin{lemma}
i) The Dirichlet kernel (\ref{Dirichlet kernel}) has the two properties
\begin{align*}
\theta_{N}\left(  x_{1},x_{2}\right)   &  =\theta_{N}\left(  x_{2}%
,x_{1}\right) \\
\theta_{N}\left(  -x_{1},x_{2}\right)   &  =\theta_{N}\left(  x_{1}%
,x_{2}\right)  .
\end{align*}

ii) If a kernel $\theta_{N}\left(  x\right)  $, $x\in T^{2}$, has these two
properties, the kernel $W_{N}=\theta_{N}\ast\theta_{N}$ has the same properties.

iii) It follows that, for any symmetric matrix $S$,%
\[
\int_{\mathbb{T}^{2}}W_{N}\left(  x\right)  \left\langle S\frac{x}{\left\vert
x\right\vert },\frac{x^{\perp}}{\left\vert x\right\vert }\right\rangle dx=0.
\]
iv) It follows also that%
\[
\lim_{N\rightarrow\infty}\int_{\mathbb{T}^{2}}\int_{\mathbb{T}^{2}}%
W_{N}\left(  x-y\right)  H_{\phi}\left(  x,y\right)  dxdy=0.
\]
In the case when $\theta_{N}$ is the Dirichlet kernel, this property is the
limit property (\ref{last limit property}).
\end{lemma}

\begin{proof}
Property (i) is obvious. The proof of (ii) is elementary, but we give the
computations for completeness:%
\begin{align*}
W_{N}\left(  x_{1},x_{2}\right)   &  =\int_{\mathbb{T}^{2}}\theta_{N}\left(
x_{1}-y_{1},x_{2}-y_{2}\right)  \theta_{N}\left(  y_{1},y_{2}\right)
dy_{1}dy_{2}\\
&  =\int_{\mathbb{T}^{2}}\theta_{N}\left(  x_{2}-y_{2},x_{1}-y_{1}\right)
\theta_{N}\left(  y_{2},y_{1}\right)  dy_{1}dy_{2}\\
&  =W_{N}\left(  x_{2},x_{1}\right)
\end{align*}%
\begin{align*}
W_{N}\left(  -x_{1},x_{2}\right)   &  =\int_{\mathbb{T}^{2}}\theta_{N}\left(
-x_{1}-y_{1},x_{2}-y_{2}\right)  \theta_{N}\left(  y_{1},y_{2}\right)
dy_{1}dy_{2}\\
&  =\int_{\mathbb{T}^{2}}\theta_{N}\left(  x_{1}+y_{1},x_{2}-y_{2}\right)
\theta_{N}\left(  y_{1},y_{2}\right)  dy_{1}dy_{2}\\
&  =\int_{\mathbb{T}^{2}}\theta_{N}\left(  x_{1}-y_{1},x_{2}-y_{2}\right)
\theta_{N}\left(  -y_{1},y_{2}\right)  dy_{1}dy_{2}\\
&  =\int_{\mathbb{T}^{2}}\theta_{N}\left(  x_{1}-y_{1},x_{2}-y_{2}\right)
\theta_{N}\left(  y_{1},y_{2}\right)  dy_{1}dy_{2}\\
&  =W_{N}\left(  x_{1},x_{2}\right)  .
\end{align*}
Let us prove (iii). We can write%
\[
\left\langle S\frac{x}{\left\vert x\right\vert },\frac{x^{\perp}}{\left\vert
x\right\vert }\right\rangle =\left(  S_{11}+S_{22}\right)  \frac{x_{1}x_{2}%
}{\left\vert x\right\vert ^{2}}+S_{12}\frac{x_{2}^{2}-x_{1}^{2}}{\left\vert
x\right\vert ^{2}}.
\]
Let us show that the integrals corresponding to each one of the two terms
vanish. We have%
\[
\int_{\mathbb{T}^{2}}W_{N}\left(  x\right)  \frac{x_{1}x_{2}}{\left\vert
x\right\vert ^{2}}dx=\int_{-\frac{1}{2}}^{\frac{1}{2}}\int_{-\frac{1}{2}%
}^{\frac{1}{2}}W_{N}\left(  x\right)  \frac{x_{1}x_{2}}{\left\vert
x\right\vert ^{2}}dx_{1}dx_{2}%
\]
The integration in the second quadrant,
\[
\int_{0}^{\frac{1}{2}}\int_{-\frac{1}{2}}^{0}W_{N}\left(  x\right)
\frac{x_{1}x_{2}}{\left\vert x\right\vert ^{2}}dx_{1}dx_{2}%
\]
cancels with the integration in the first quadrant,%
\[
\int_{0}^{\frac{1}{2}}\int_{0}^{\frac{1}{2}}W_{N}\left(  x\right)  \frac
{x_{1}x_{2}}{\left\vert x\right\vert ^{2}}dx_{1}dx_{2}%
\]
because of property $W_{N}\left(  -x_{1},x_{2}\right)  =W_{N}\left(
x_{1},x_{2}\right)  $ (point (ii)); similarly for the integrations in the
other quadrants. So $\int_{\mathbb{T}^{2}}W_{N}\left(  x\right)  \frac
{x_{1}x_{2}}{\left\vert x\right\vert ^{2}}dx=0$. For the other integral, just
by renaming the variables we have
\[
\int_{\mathbb{T}^{2}}W_{N}\left(  x_{1},x_{2}\right)  \frac{x_{1}^{2}%
}{\left\vert x\right\vert ^{2}}dx_{1}dx_{2}=\int_{\mathbb{T}^{2}}W_{N}\left(
x_{2},x_{1}\right)  \frac{x_{2}^{2}}{\left\vert x\right\vert ^{2}}dx_{2}dx_{1}%
\]
and then, using $W_{N}\left(  x_{1},x_{2}\right)  =W_{N}\left(  x_{2}%
,x_{1}\right)  $ (point (ii))%
\[
=\int_{\mathbb{T}^{2}}W_{N}\left(  x_{1},x_{2}\right)  \frac{x_{2}^{2}%
}{\left\vert x\right\vert ^{2}}dx_{1}dx_{2}%
\]
hence $\int_{\mathbb{T}^{2}}W_{N}\left(  x\right)  \frac{x_{2}^{2}-x_{1}^{2}%
}{\left\vert x\right\vert ^{2}}dx=0$. We have proved (iii).

Finally, the limit in (iv) is a consequence of the decompositon
(\ref{decomposition H}). Indeed,
\begin{align*}
&  \int_{\mathbb{T}^{2}}\int_{\mathbb{T}^{2}}W_{N}\left(  x-y\right)
\left\langle D^{2}\phi\left(  x\right)  \frac{x-y}{\left\vert x-y\right\vert
},\frac{\left(  x-y\right)  ^{\perp}}{\left\vert x-y\right\vert }\right\rangle
dxdy\\
&  =\int_{\mathbb{T}^{2}}\left(  \int_{\mathbb{T}^{2}}W_{N}\left(  z\right)
\left\langle D^{2}\phi\left(  x\right)  \frac{z}{\left\vert z\right\vert
},\frac{z^{\perp}}{\left\vert z\right\vert }\right\rangle dz\right)  dx=0
\end{align*}
by (iii), and%
\[
\lim_{N\rightarrow\infty}\int_{\mathbb{T}^{2}}\int_{\mathbb{T}^{2}}%
W_{N}\left(  x-y\right)  R_{\phi}\left(  x,y\right)  dxdy=0
\]
because $R_{\phi}\left(  x,y\right)  $ is Lipschitz continuous with
$\left\vert R_{\phi}\left(  x,y\right)  \right\vert \leq C\left\vert
x-y\right\vert $. To complete the proof of the claims of part (iv), let us
check that, when $\theta_{N}$ is the Dirichlet kernel, the property stated in
(iv) coincides with the limit property (\ref{last limit property}). We have
\begin{align*}
\int_{\mathbb{T}^{2}}\left(  H_{\phi}\right)  _{N}\left(  x,x\right)  dx &
=\sum_{\left\vert n^{\prime}\right\vert _{\infty}\leq N}^{N}\sum_{\left\vert
n\right\vert _{\infty}\leq N}^{N}\int_{\mathbb{T}^{2}}\int_{\mathbb{T}^{2}%
}\int_{\mathbb{T}^{2}}e^{2\pi in^{\prime}\cdot\left(  x-x^{\prime}\right)
}e^{2\pi in\cdot\left(  x-y^{\prime}\right)  }H_{\phi}\left(  x^{\prime
},y^{\prime}\right)  dy^{\prime}dx^{\prime}dx\\
&  =\int_{\mathbb{T}^{2}}\int_{\mathbb{T}^{2}}\left(  \sum_{\left\vert
n^{\prime}\right\vert _{\infty}\leq N}^{N}\sum_{\left\vert n\right\vert
_{\infty}\leq N}^{N}\int_{\mathbb{T}^{2}}e^{2\pi in^{\prime}\cdot\left(
x^{\prime}-x\right)  }e^{2\pi in\cdot\left(  x-y^{\prime}\right)  }dx\right)
H_{\phi}\left(  x^{\prime},y^{\prime}\right)  dy^{\prime}dx^{\prime}\\
&  =\int_{\mathbb{T}^{2}}\int_{\mathbb{T}^{2}}W_{N}\left(  x^{\prime
}-y^{\prime}\right)  H_{\phi}\left(  x^{\prime},y^{\prime}\right)  dy^{\prime
}dx^{\prime}.
\end{align*}
The proof is complete.
\end{proof}

\subsubsection{General case: $\rho_{0}$ of class LlogL}

Assume now that $\rho_{0}$ satisfies only the assumptions of the main theorem.
By Corollory C.3 in \cite{DaPratoRoeckner}, there exists a sequence $\rho
_{0}^{n}$ of bounded continuous functions (in fact bounded smooth cylinder
functions) that converges to $\rho_{0}$ in $L^{1}\left(  B,\mu\right)  $ and
\[
C:=\sup_{n\in\mathbb{N}}\int_{B}\rho_{0}^{n}\left(  \omega\right)  \log
\rho_{0}^{n}\left(  \omega\right)  \mu\left(  d\omega\right)  <\infty.
\]
For each $n$, apply the result of the first case and construct a weak solution
$\rho_{t}^{n}$, which fulfills in particular
\[
\int_{B}\rho_{t}^{n}\left(  \omega\right)  \log\rho_{t}^{n}\left(
\omega\right)  \mu\left(  d\omega\right)  \leq\int_{B}\rho_{0}^{n}\left(
\omega\right)  \log\rho_{0}^{n}\left(  \omega\right)  \mu\left(
d\omega\right)  \leq C.
\]
From this inequality we deduce the existence of a subsequence, still denoted
for simplicity by $\rho_{t}^{n}\left(  \omega\right)  $ which converges to
some function $\rho_{t}$ weak* in $L^{1}\left(  0,T;L^{1}\left(  B,\mu\right)
\right)  $, which satisfies property (\ref{regularity rho}), and moreover,
from the duality of Orlicz spaces, such that%
\[
\int_{0}^{T}\int_{B}G\left(  t,\omega\right)  \rho_{t}^{n}\left(  \pi
_{N}\omega\right)  \mu\left(  d\omega\right)  dt\rightarrow\int_{0}^{T}%
\int_{B}G\left(  t,\omega\right)  \rho_{t}\left(  \omega\right)  \mu\left(
d\omega\right)  dt
\]
for all $G$ such that, for some $\epsilon>0$,
\begin{equation}
\sup_{t\in\left[  0,T\right]  }\int_{B}e^{\epsilon\left\vert G\left(
t,\omega\right)  \right\vert }\mu\left(  d\omega\right)  <\infty
.\label{exp integrab test function}%
\end{equation}
Due to these fact, in order to prove that $\rho_{t}$ satisfies the weak
formulation of the continuity equation, we have only to prove that
\[
\int_{0}^{T}\int_{B}\left\langle b\left(  \omega\right)  ,DF\left(
t,\omega\right)  \right\rangle \rho_{t}^{n}\left(  \omega\right)  \mu\left(
d\omega\right)  dt\rightarrow\int_{0}^{T}\int_{B}\left\langle b\left(
\omega\right)  ,DF\left(  t,\omega\right)  \right\rangle \rho_{t}\left(
\omega\right)  \mu\left(  d\omega\right)  dt.
\]
Since $G\left(  t,\omega\right)  :=\left\langle b\left(  \omega\right)
,DF\left(  t,\omega\right)  \right\rangle $ has property
(\ref{exp integrab test function}) by Theorem \ref{thm exponential est}, this
is true, and the proof is complete.

\begin{acknowledgement}
G. Da Prato and F. Flandoli are partially supported by GNAMPA from INDAM and
by MIUR, PRIN projects 2015. M. R\"{o}ckner is supported by CRC 1283 through
the DFG.
\end{acknowledgement}

\end{document}